\documentclass[12pt]{amsart}

\usepackage{amsmath}
\usepackage{amssymb}
\usepackage{amsaddr}
\usepackage{amsthm}
\usepackage{diagbox}

\newtheorem{theorem}{Theorem}
\newtheorem{prop}{Proposition}
\newtheorem{lemma}{Lemma}

\def\F{\mathbb{F}}
\def\ndiv{\nmid}
\def\eps{\varepsilon}
\def\sumrhochi{\sum_{\rho\text{ of }L_\chi}}

\DeclareMathOperator{\chr}{char}

\title{The Counting function for Elkies primes}
\author{Meher Elijah Lippmann and Kevin J.~McGown}
\address[A1,A2]{California State University, Chico, USA}

\begin{document}

\maketitle

\begin{abstract}
Let $E$ be an elliptic curve over a finite field $\mathbb{F}_q$ where $q$ is a prime power. 
The Schoof--Elkies--Atkin (SEA) algorithm is a standard method for counting the number of
$\mathbb{F}_q$-points on $E$.
The asymptotic complexity of the SEA algorithm depends on the distribution of the so-called Elkies primes.

Assuming GRH, we prove that the least Elkies prime is bounded by
$(2\log 4q+4)^2$ when $q\geq 10^9$.
This is the first such explicit bound in the literature.
Previously, Satoh and Galbraith established an upper bound of $O((\log q)^{2+\eps})$.

Let $N_E(X)$ denote the number of Elkies primes less than $X$.
Assuming GRH, we also show
$$
  N_E(X)=\frac{\pi(X)}{2}+O\left(\frac{\sqrt{X}(\log qX)^2}{\log X}\right)
  \,.
$$

\end{abstract}

\section{Introduction}\label{S:1}

Let $E$ be an elliptic curve over $\F_q$ where $q$ is a prime power.
It is of high interest to compute $\#E(\F_q)$, 
where $E(\F_q)$ denotes the group of $\F_q$ points on $E$.
This is of theoretical significance in addition to its importance for cryptographic applications.
The Schoof–Elkies–Atkin (SEA) algorithm is a standard method for
accomplishing this task.
For more details see~\cite{MR0777280},~\cite{MR1486831},
\mbox{\cite[Algorithm 17.25]{MR2162716}}.
One observes that the asymptotic time complexity of the algorithm depends
dramatically  on the distribution of the so-called Elkies primes.

We define $t_E=q+1-\#E(\F_q)$;
it is well known that this is the trace of the Frobenius endomorphism
on $E$.  We refer to a prime $\ell$ with
$(\ell,q)=1$
as an Elkies prime 
if the Legendre symbol $(\frac{t_E^2-4q}{\ell})$ equals $1$.
Primes
$(\ell,q)=1$
where the symbol equals $-1$ or $0$ are called Atkin primes.

If $t_E^2-4q=0$, then there are no Elkies primes; this is the case where $|t|=2\sqrt{q}$
and hence the curve is ``strongly maximal'' with respect to the Hasse bound.
However, this latter condition can only happen for supersingular curves.
We will assume throughout that $t_E^2-4q\neq 0$.
Notice that if $t_E^2-4q$ is a perfect square, then
any prime $\ell$ with $(\ell,t_E^2-4q)=1$ is an Elkies prime.  We will avoid
this trivial case by assuming $t_E^2-4q$ is not a perfect square.

In~\cite{MR2041073}, it is shown that the least Elkies prime is
$O((\log q)^{2+\eps})$.
However, it may be that explicit results are more relevant for some applications.
We give an explicit result that also offers a small asymptotic improvement.
Moreover, our proof is different from the one appearing in~\cite{MR2041073}
as we follow the approach of Montgomery--Lagarias--Odlyzko--Bach, rather
than the arguments of Ankeny.

\begin{theorem}\label{T:1}
Assume GRH.  For $q\geq 10^9$,
the least Elkies prime is bounded above by
$(2\log 4q+4)^2$.
In particular, the least Elkies prime is $O((\log q)^2)$.
\end{theorem}

Let $N_E(X)$ denote the number of Elkies primes up to $X$.
It is stated in Appendix A of~\cite{MR2041073} that estimating $N_E(X)$ seems to be difficult
(even under GRH).  We give such an estimate with error term, but will not
pursue making the constant explicit.

\begin{theorem}\label{T:2}
Assume GRH.  We have
$$
  N_E(X)=\frac{\pi(X)}{2}+O\left(\frac{\sqrt{X}(\log qX)^2}{\log X}\right)
  \,.
$$

\end{theorem}

Notice that the previous theorem immediately implies (under GRH)
that $N_E(X)\sim\pi(X)/2$, as expected.
It is noted in~\cite{MR3179585} that this corollary follows from~\cite{MR2041073};
however, neither paper gives a bound on the error $N_E(X)-\pi(X)/2$.
We note that the paper~\cite{MR3179585} studies the distribution of Elkies primes
on average over all elliptic curves, which is a different perspective than we consider here.

\section{A character sum estimate}

We will require the next proposition in the sequel.
It follows from the work of Bach
(see~\cite{MR807772, MR1023756, MR1433261}) but a little effort is needed for its proof.
Notice that the modulus of the character is not required to be prime.

\begin{prop}\label{prop:charsum}
Assume GRH.
\begin{enumerate}
\item
When $X\geq X_0\geq 2$ we have
$$
  \left|
  \sum_{n\leq X}\Lambda(n)\left(1-\frac{n}{X}\right)-\frac{X}{2}
  \right| \leq C_1(X_0)\sqrt{X}
  \,.
$$
\item
Suppose $\chi$ is a non-principal primitive character modulo $m$.
When \mbox{$X>X_0\geq 10$} and $m\geq m_0$ we have

$$
  \left|
  \sum_{n\leq X}\Lambda(n)\chi(n)\left(1-\frac{n}{X}\right)
  \right| \leq C_2(m_0,X_0) \sqrt{X}\log m
  \,,
$$
where values of the constants are given in Tables~\ref{tab:1},~\ref{tab:2},~\ref{tab:3},
generated by Equations (\ref{E:const1}), (\ref{E:const2}), (\ref{E:const3}).
\end{enumerate}
\end{prop}

\begin{table}[h]
 \footnotesize{
  \begin{tabular}{|c|c|c|c|c|c|c|c|c|c|}
    \hline
    $x_0$ & $10^{1}$ & $2^{4}$ & $2^{5}$ & $2^{6}$ & $2^{7}$ & $2^{8}$ & $2^{9}$ & $2^{10}$ & $2^{11}$  \\
    \hline
    $C_1$ & $0.629$ & $0.507$ & $0.048$ & $0.047$ & $0.047$ & $0.047$ & $0.047$ & $0.047$ & $0.047$  \\
    \hline    
  \end{tabular}
  \caption{Values of $C_1(X_0)$}
  \label{tab:1}
  }
\end{table}
\begin{table}[h]
\footnotesize{
  \begin{tabular}{|c|c|c|c|c|c|c|c|c|c|c|c|}
    \hline
    \diagbox[dir=SE]{{$x_0$}}{{$C_2$}}{{$m_0$}} & $3$ & $2^{2}$ & $2^{3}$ & $2^{4}$ & $2^{5}$ & $2^{6}$ & $2^{7}$ & $2^{8}$ & $2^{9}$ & $2^{10}$ & $2^{11}$  \\
    \hline
    $10$ & $3.475$ & $2.993$ & $2.379$ & $2.072$ & $1.888$ & $1.766$ & $1.678$ & $1.612$ & $1.561$ & $1.520$ & $1.487$  \\
    \hline
    $2^{4}$ & $3.051$ & $2.632$ & $2.099$ & $1.832$ & $1.672$ & $1.565$ & $1.489$ & $1.432$ & $1.387$ & $1.352$ & $1.323$  \\
    \hline
    $2^{5}$ & $2.621$ & $2.267$ & $1.816$ & $1.591$ & $1.455$ & $1.365$ & $1.301$ & $1.252$ & $1.215$ & $1.185$ & $1.160$  \\
    \hline
    $2^{6}$ & $2.335$ & $2.025$ & $1.629$ & $1.431$ & $1.313$ & $1.233$ & $1.177$ & $1.134$ & $1.101$ & $1.075$ & $1.054$  \\
    \hline
    $2^{7}$ & $2.140$ & $1.859$ & $1.501$ & $1.323$ & $1.215$ & $1.144$ & $1.093$ & $1.054$ & $1.025$ & $1.001$ & $0.981$  \\
    \hline
    $2^{8}$ & $2.003$ & $1.743$ & $1.413$ & $1.247$ & $1.148$ & $1.082$ & $1.035$ & $0.999$ & $0.972$ & $0.949$ & $0.931$  \\
    \hline
    $2^{9}$ & $1.908$ & $1.663$ & $1.350$ & $1.194$ & $1.101$ & $1.038$ & $0.994$ & $0.960$ & $0.934$ & $0.914$ & $0.897$  \\
    \hline
    $2^{10}$ & $1.840$ & $1.606$ & $1.307$ & $1.157$ & $1.068$ & $1.008$ & $0.965$ & $0.933$ & $0.908$ & $0.888$ & $0.872$  \\
    \hline
    $2^{11}$ & $1.793$ & $1.565$ & $1.276$ & $1.131$ & $1.044$ & $0.986$ & $0.945$ & $0.914$ & $0.890$ & $0.870$ & $0.855$  \\
    \hline       
  \end{tabular}
  \caption{Values of $C_2(X_0,m_0)$ for $\chi(-1)=-1$}
  \label{tab:2}
  }
\end{table}
\begin{table}[h]
\footnotesize{
  \begin{tabular}{|c|c|c|c|c|c|c|c|c|c|c|c|}
    \hline
    \diagbox[dir=SE]{{$x_0$}}{{$C_2$}}{{$m_0$}} & $3$ & $2^{2}$ & $2^{3}$ & $2^{4}$ & $2^{5}$ & $2^{6}$ & $2^{7}$ & $2^{8}$ & $2^{9}$ & $2^{10}$ & $2^{11}$  \\
    \hline
    $10$ & $3.850$ & $3.290$ & $2.577$ & $2.221$ & $2.007$ & $1.865$ & $1.763$ & $1.686$ & $1.627$ & $1.580$ & $1.541$  \\
    \hline
    $2^{4}$ & $3.455$ & $2.952$ & $2.312$ & $1.992$ & $1.800$ & $1.672$ & $1.580$ & $1.512$ & $1.458$ & $1.416$ & $1.381$  \\
    \hline
    $2^{5}$ & $3.018$ & $2.582$ & $2.026$ & $1.748$ & $1.581$ & $1.470$ & $1.390$ & $1.331$ & $1.284$ & $1.247$ & $1.217$  \\
    \hline
    $2^{6}$ & $2.695$ & $2.309$ & $1.819$ & $1.574$ & $1.426$ & $1.328$ & $1.258$ & $1.206$ & $1.165$ & $1.132$ & $1.105$  \\
    \hline
    $2^{7}$ & $2.449$ & $2.104$ & $1.665$ & $1.445$ & $1.314$ & $1.226$ & $1.163$ & $1.116$ & $1.079$ & $1.050$ & $1.026$  \\
    \hline
    $2^{8}$ & $2.262$ & $1.948$ & $1.549$ & $1.350$ & $1.230$ & $1.150$ & $1.093$ & $1.050$ & $1.017$ & $0.990$ & $0.969$  \\
    \hline
    $2^{9}$ & $2.118$ & $1.830$ & $1.462$ & $1.278$ & $1.168$ & $1.094$ & $1.041$ & $1.002$ & $0.971$ & $0.947$ & $0.927$  \\
    \hline
    $2^{10}$ & $2.009$ & $1.739$ & $1.396$ & $1.224$ & $1.121$ & $1.052$ & $1.003$ & $0.967$ & $0.938$ & $0.915$ & $0.896$  \\
    \hline
    $2^{11}$ & $1.926$ & $1.671$ & $1.346$ & $1.184$ & $1.086$ & $1.021$ & $0.975$ & $0.940$ & $0.913$ & $0.892$ & $0.874$  \\
    \hline       
  \end{tabular}
  \caption{Values of $C_2(X_0,m_0)$ for $\chi(-1)=1$}
  \label{tab:3}
  }
\end{table}

\begin{proof}
Let $X\geq 2$.
Using the explicit formula, in the form of
Lemma~\ref{L:explicit},
we find
\begin{align*}
  \left|\sum_{n\leq X}\Lambda(n)(1-n/X)-\frac{X}{2}\right|
  &\leq
  \sqrt{X}\sum_{\substack{\zeta(\rho)=0\\\Re(\rho)=1/2}}\left|\frac{1}{\rho(\rho+1)}\right|+1.84
  \,.
\end{align*}
Here we have used
$$
  \left|\frac{1}{X}\cdot\frac{\zeta'(-1)}{\zeta(-1)}-\frac{\zeta'(0)}{\zeta(0)}
  -\frac{1}{2}\log\left(1-\frac{1}{X^2}\right)-\frac{1}{2x}\left(1+\frac{2}{X-1}\right)
  \right|\leq\frac{\zeta'(0)}{\zeta(0)}\leq 1.84
  \,.
$$
Combining with the following (see Equations 10 and~11 from \S $12$ of~\cite{MR1790423})
$$
  \sum_\rho\left|\frac{1}{\rho(\rho+1)}\right|
  \leq
  \sum_\rho
  \left(
  \frac{1}{\rho}+\frac{1}{\overline{\rho}}\right)=\gamma+2-\log(4\pi)
$$
gives part (1) of the proposition with
\begin{equation}\label{E:const1}
  C_1=\gamma+2-\log(4\pi)+\frac{1.84}{\sqrt{X_0}}
  \,.
\end{equation}
It remains to prove part (2).

We first deal with the case $\chi(-1)=-1$.
Lemma~\ref{L:explicit}
gives
\begin{align*}
\left|\sum_{n\leq X}\Lambda(n)\chi(n)(1-n/X)\right|
&\leq
\sqrt{X}\sum_{\substack{L(\rho,\chi)=0\\\Re(\rho)=1/2}}
\left|\frac{1}{\rho(\rho+1)}\right|
+\left|\frac{L'(0,\chi)}{L(0,\chi)}\right|+\left|\frac{b_\chi}{X}\right|
\\[1ex]
&
+\left|\frac{\log X}{X}+\frac{1}{2}\log\left(1+\frac{2}{X-1}\right)+\frac{1}{2x}\log\left(1-\frac{1}{X^2}\right)\right|
\,.
\end{align*}
Note that
\begin{align*}
\left|\frac{b_\chi}{X}\right|
&\leq \frac{1}{X}\left|\frac{L'(2,\chi)}{L(2,\chi)}\right|+\frac{1}{X}\sum\frac{1}{|\rho(\rho+1)|}+\frac{2}{X}\sum\frac{1}{|\rho(2-\rho)|}
+\left|\frac{\log 2-1}{X}\right|
\\[1.5ex]
&\leq
\frac{0.87}{X}+\frac{3}{X}
\sum_{\substack{L(\rho,\chi)=0\\\Re(\rho)=1/2}}
\left|\frac{1}{\rho(\rho+1)}\right|
\end{align*}
and
\begin{align}
\label{E:temp1}
\left|
\frac{L'(0,\chi)}{L(0,\chi)}
\right|
&=
\left|-\sum\frac{2}{\rho(2-\rho)}+\frac{L'(2,\chi)}{L(2,\chi)}-1\right|
\leq
2\sum_{\substack{L(\rho,\chi)=0\\\Re(\rho)=1/2}}
\left|\frac{1}{\rho(\rho+1)}\right|
+1.57
\,.
\end{align}
For $X\geq 10$, this all leads to
$$
\left|\sum_{n\leq X}\Lambda(n)\chi(n)(1-n/X)\right|
\leq
(\sqrt{X}+2+3/X)\sum_{\substack{L(\rho,\chi)=0\\\Re(\rho)=1/2}}
\left|\frac{1}{\rho(\rho+1)}\right|
+2
$$
Using estimates from~\cite{MR1433261} we have
\begin{equation}\label{E:temp2}
\sum_{\substack{L(\rho,\chi)=0\\\Re(\rho)=1/2}}
\left|\frac{1}{\rho(\rho+1)}\right|
\leq
\frac{4}{3}
\sum_\rho\frac{1}{|\sigma-\rho|^2}=\sum_\rho\left(\frac{1}{\sigma-\rho}+\frac{1}{\sigma-\overline{\rho}}\right)
\leq
\frac{2}{3}(\log m+5/3)
\end{equation}
and the second result follows with 
\begin{equation}\label{E:const2}
  C_2=\frac{2}{3}\left(1+\frac{2}{\sqrt{X_0}}+\frac{3}{X_0\sqrt{X_0}}\right)\left(1+\frac{5}{3\log m_0}\right)+\frac{2}{\sqrt{X_0}\log m_0}
  \,.
\end{equation}

Next we turn to the case of $\chi(-1)=1$.
The proof is similar.  We have
\begin{align*}
\left|\sum_{n\leq X}\Lambda(n)\chi(n)(1-n/X)\right|
&\leq
\sqrt{X}\sum_{\substack{L(\rho,\chi)=0\\\Re(\rho)=1/2}}
\left|\frac{1}{\rho(\rho+1)}\right|
+\left|\frac{1}{X}\frac{L'(-1,\chi)}{L(-1,\chi)}\right|+\left|b_\chi\right|
\\
&
+\log X+\left|\frac{1}{2x}\log\left(1+\frac{2}{X-1}\right)+\frac{1}{2}\log\left(1-\frac{1}{X^2}\right)\right|
\,,
\end{align*}
and we make use of
\begin{align*}
\left|b_\chi\right|
&\leq
\sqrt{X}\sum_{\substack{L(\rho,\chi)=0\\\Re(\rho)=1/2}}
\left|\frac{1}{\rho(\rho+1)}\right|+0.57
\,,
\\
\left|
\frac{L'(-1,\chi)}{L(-1,\chi)}
\right|
&\leq
\sum\frac{2}{|\rho(2-\rho)|}
+
\sum\frac{1}{|\rho(\rho+1)|}
+
\frac{1}{2}
\left|\frac{\Gamma'(1)}{\Gamma(1)}-\frac{\Gamma'(-1/2)}{\Gamma(-1/2)}\right|
+\left|
\frac{L'(2,\chi)}{L(2,\chi)}
\right|
\\
&\leq
3\sum\frac{1}{|\rho(\rho+1)|}
+0.88
\,.
\end{align*}
For $X\geq 3$, this all leads to
$$
\left|\sum_{n\leq X}\Lambda(n)\chi(n)(1-n/X)\right|
\leq
(\sqrt{X}+2+3/X)\sum_{\substack{L(\rho,\chi)=0\\\Re(\rho)=1/2}}
\left|\frac{1}{\rho(\rho+1)}\right|
+\log X +1
\,,
$$
which gives the result with
\begin{equation}\label{E:const3}
  C_2=\frac{2}{3}\left(1+\frac{2}{\sqrt{X_0}}+\frac{3}{X_0\sqrt{X_0}}\right)\left(1+\frac{5}{3\log m_0}\right)+\frac{1+\log X_0}{\sqrt{X_0}\log m_0}
  \,.
\end{equation}

\end{proof}

\section{Proof of Theorem~\ref{T:1}}

In light of our assumptions (stated in Section~\ref{S:1})
we can use the Kronecker symbol to define a Dirichlet character
$$\chi_E(n)=\left(\frac{t_E^2-4q}{n}\right)\,.$$ 
Moreover, $\chi_E$ is a non-principal character modulo $m=|t_E^2-4q|\geq 3$
satisfying $\chi_E(-1)=-1$.
Let $\chi$ denote the primitive character of modulus $m'$ that induces $\chi_E$.
Notice that $m'\leq m\leq 4q$.  
Set $X=(2\log 4q+4)^2\geq 2325$.
We will show that there is an Elkies prime less than $X$.
By direct computation, there is always an Elkies prime less than $40$ when $m< 2^{11}$,
and hence the theorem is trivial in this case.  Henceforth, we assume $m\geq 2^{11}$.

Since our character $\chi_E$ of modulus $m$ is not primitive,
the difference between the sum
$$
  \sum_{n\leq X}\Lambda(n)\chi_E(n)\left(1-\frac{n}{X}\right)
$$
and the corresponding sum where $\chi_E$ is replaced by the primitive
character $\chi$ is bounded (in absolute value) by
$$
  \sum_{p|m}\sum_{\substack{n\leq X\\n=p^k}}\Lambda(n)\left(1-\frac{n}{X}\right)
  \leq
  \sum_{p|m}\log p\sum_{\substack{n\leq X\\n=p^k}} 1\leq \omega(m)\log X
  \leq D
  \frac{\log m\log X}{\log\log m}
$$
By~\cite{MR736719}, we know that for $m\geq 3$, one has
$$
\omega(m)\leq 1.3841\frac{\log m}{\log \log m}
$$
which means we can take $D=1.3841$.


Now we begin the proof in earnest.
Suppose that the least Elkies prime is greater than $X$.
That is, we have that $\chi_E(p)\neq 1$ for all primes $p\leq X$.
We seek a contradiction.
Consider the sum
$$
  S(X)=\sum_{n\leq X}\Lambda(n)\chi_E(n)\left(1-\frac{n}{X}\right)
$$
We write $S=S_1+S_2+S_3$, where
\begin{align*}
  S_1(X)&=
  -\sum_{n\leq X}\Lambda(n)\left(1-\frac{n}{X}\right)
  \\
    S_2(X)&=
  \sum_{p\leq X}\Lambda(p)(\chi_E(p)+1)\left(1-\frac{p}{X}\right)
  \\
  S_3(X)&=
  \sum_{\substack{n\leq X\\n=p^k\\k\geq 2}}\Lambda(n)(\chi_E(n)+1)\left(1-\frac{n}{X}\right)
\end{align*}
By hypothesis,
$$
  S_2(X)\leq\sum_{p|m}\log p\left(1-\frac{p}{X}\right)\leq \log(\prod_{p|m}p)\leq \log m
$$
Moreover,
according to
Corollary~4.5 of~\cite{MR3745073},
$$
\psi(X)-\theta(X)\leq  (1+1.47\cdot 10^{-7})\sqrt{X}+1.78\sqrt[3]{X}+\leq (1+\delta)\sqrt{X}
$$
where $\delta=1/2$,
and therefore
$$
|S_3(X)|\leq 2(\psi(X)-\theta(X))\leq 2(1+\delta)\sqrt{X}
\,.
$$
By Proposition~\ref{prop:charsum}
$$
  |S_1|\geq \frac{X}{2}-C_1\sqrt{X}
$$
Therefore
\begin{align*}
|S(X)|
&\geq
\frac{X}{2}-(C_1+2(1+\delta))\sqrt{X}-\log m
\end{align*}
On the other hand,
$$
  |S(X)|\leq C_2\sqrt{X}\log m+ D\frac{\log m\log X}{\log\log m}
$$
Hence
$$
  \frac{X}{2}\leq C_2\sqrt{X}\log m+(C_1+2(1+\delta))\sqrt{X}+\log m+ D\frac{\log m\log X}{\log\log m}
$$
Using $m\leq 4q$, this leads to
\begin{align*}
\sqrt{X}
&\leq 2 C_2\log 4q + 2(C_1+2(1+\delta))+\frac{2\log 4q}{\sqrt{X}}+\frac{2D\log(4q)\log X}{\sqrt{X}\log\log 4q}
\\
&\leq
1.902\log 4q+6.094
\\[1ex]
&<
2\log(4q) + 4
.
\end{align*}
This gives a contradiction,
thereby proving the theorem.

\section{Proof of Theorem~\ref{T:2}}

We have
\begin{align*}
  N_E(X)
  &=
  \frac{1}{2}\sum_{\ell\leq X}\left[\left(\frac{t_E^2-4q}{\ell}\right)+1\right]
  -a_E(X)
  \\[1ex]
  &=
  \frac{\pi(X)}{2}+\frac{1}{2}
  \sum_{\ell\leq X}\left(\frac{t_E^2-4q}{\ell}\right)
  -a_E(X)
\end{align*}
In the above
$$
a_E(X)=\begin{cases}
1/2 & \ell\mid t\,,\;\chr(\F_q)<X\\
1 & \ell\ndiv t\,,\;\chr(\F_q)<X\\
0 & \text{otherwise}
\end{cases}
$$

Using the notation from the previous section, it suffices to bound the sum
$$
\sum_{\ell\leq X}\chi_E(\ell)
\,.
$$

Partial summation gives
$$
  \sum_{\ell\leq X}\chi(\ell)=\frac{1}{\log X}\sum_{\ell\leq X}\chi(\ell)\Lambda(\ell)
  +\int_2^X \left(\sum_{\ell\leq t}\chi(\ell)\Lambda(\ell)\right)\frac{dt}{t(\log t)^2}
  \,.
$$
Exercise 7.4.9 of~\cite{MR2376618} leads to a bound of the form:
\begin{equation}\label{E:bound}
\left|\sum_{\ell\leq X}\chi(\ell)\Lambda(\ell)\right|\leq C \sqrt{X}(\log mX)^2
\,,\; t>2\,.
\end{equation}

This leads to
\begin{align*}
  \left|\sum_{\ell\leq X}\chi(\ell)\right|
  &\leq
  \frac{C\sqrt{X}(\log mX)^2}{\log X}
  +C\int_2^X \frac{(\log tm)^2dt}{\sqrt{t}(\log t)^2}
  \,.
\end{align*}
Using $X<m$ the second term is bounded by
$4C(\log m)^2\int_2^xt^{-1/2}(\log t)^{-2}dt$.
Moreover, we have
$$
  \int_2^X\frac{dt}{\sqrt{t}(\log t)^2}\leq 2\frac{\sqrt{X}}{\log X}
  \,.
$$
This yields
\begin{align*}
  \frac{C\sqrt{X}(\log mX)^2}{\log X}
  +\frac{8C\sqrt{X}(\log m)^2}{\log X}
  &=
  C((\log mX)^2+8(\log m)^2)\frac{\sqrt{X}}{\log X}
  \\
  &\leq
  9C(\log mX)^2\frac{\sqrt{X}}{\log X}
\end{align*}
Putting this all together shows
$$
  \left|
  N_E(X)-\frac{\pi(X)}{2}
  \right|\leq
  (9C/2)\left(\frac{\sqrt{X}(\log 4qX)^2}{\log X}\right)+1
  \,,
$$
which completes the proof.


\section{Appendix:  An explicit formula}

Although various explicit formulas associated to Dirichlet $L$-functions are certainly well-known,
we could not find in the literature this particular explicit formula, with all terms rendered visible, for
an arbitrary Dirichlet character.  As this is essential for the derivation of our results, we give a brief
argument together with relevant residue calculations.
The first of the three equations that follow appears in~\cite{MR1074573}.

\begin{lemma}\label{L:explicit}
Suppose $X\geq 1$.
We have
\begin{align*}
  \sum_{n\leq X}'\Lambda(n)\left(1-\frac{n}{X}\right)
  &=
  \frac{X}{2}
   -\sum_\rho\frac{X^{\rho}}{\rho(\rho+1)}
   +
     \frac{1}{X}\frac{\zeta'(-1)}{\zeta(-1)}-\frac{\zeta'(0)}{\zeta(0)}
   \\
   &-
      \frac{1}{2}\log\left(1-\frac{1}{X^2}\right)-\frac{1}{2x}\log\left(1+\frac{2}{X-1}\right)
      \,.
\end{align*}
Suppose $\chi$ is a primitive Dirichlet character modulo $m$.
If  $\chi(-1)=-1$, then we have
\begin{align*}
  \sum_{n\leq X}\Lambda(n)\chi(n)\left(1-\frac{n}{X}\right)
  &=
   -\sumrhochi
   \frac{X^{\rho}}{\rho(\rho+1)}
   - \frac{L'(0)}{L(0)}
   +
     \frac{b_\chi}{X}+
  \frac{\log X}{X}
   \\+&
  \frac{1}{2}\log\left(1+\frac{2}{X-1}\right)+\frac{1}{2x}\log\left(1-\frac{1}{X^2}\right)
  \,,
\end{align*}
where $b_\chi$ is the constant term in the Laurent series expansion of
$L'(s,\chi)/L(s,\chi)$ at $s=-1$.
If  $\chi(-1)=1$, then we have
\begin{align*}
  \sum_{n\leq X}\Lambda(n)\chi(n)\left(1-\frac{n}{X}\right)
  &=
   -\sumrhochi
   \frac{X^{\rho}}{\rho(\rho+1)}
   +\frac{1}{X}\frac{L'(-1)}{L(-1)}-b_\chi-\log X
  \\
  &
  -\frac{1}{2}\log\left(1-\frac{1}{X^2}\right)-\frac{1}{2x}\log\left(1+\frac{2}{X-1}\right)
  \,,
\end{align*}
where $b_\chi$ is the constant term in the Laurent series
expansion of $L'(s,\chi)/L(s,\chi)$ at $s=0$.

\end{lemma}

\begin{proof}
We begin by proving the first equation that involves the zeros of $\zeta(s)$.
We are interested in the sum
\begin{equation}
  g(X):=\sum_{n\leq X}\Lambda(n)(X-n)
  \,.
\end{equation}
Along these lines, we consider the function
\begin{align*}
  G(s):=\sum_{n=1}^\infty\frac{\Lambda(n)(X-n)}{n^s}
  &=
  X\sum_n\frac{\Lambda(n)}{n^s}-\sum_n\frac{\Lambda(n)}{n^{s-1}}
  &=
  X\frac{\zeta'(s)}{\zeta(s)}-\frac{\zeta'(s-1)}{\zeta(s-1)}
  \,.
\end{align*}
We evaluate the following integral in two different ways:
$$
\mathcal{I}(X):=
  \frac{1}{2\pi i}
    \int_{\sigma-i\infty}^{\sigma+i\infty}\left\{-G(s)\right\}X^s\frac{ds}{s}
    \,.
$$
Integrating the Dirichlet series for $G(s)$ term-by-term and using
$$
  \int_{\sigma-i\infty}^{\sigma+i\infty}y^s\;\frac{ds}{s}
  =\begin{cases} 1 & y>1\\ 1/2 & y=1\\ 0 & 0<y<1\,\end{cases}
$$
yields $\mathcal{I}(X)=g_0(X)$.  Here $g_0(X)=g(X)$ when $X$ is non-integral.
Next we evaluate the integral using the Residue Theorem. 
Standard arguments show that $\mathcal{I}(X)$ is
the sum of the residues in the half-plane where $\Re(s)<\sigma$.
The residues of the integrand at $s=1$ and $s=2$
are $X^2$ and $-X^2/2$, respectively.
The residue at $s=0$ is
$$
  \frac{\zeta'(-1)}{\zeta(-1)}-X\frac{\zeta'(0)}{\zeta(0)}
  \,.
$$
Given a non-trivial zero $\rho$, the residues at $s=\rho$ and $s=\rho+1$
are $-X^{\rho+1}/\rho$
and $X^{\rho+1}/(\rho+1)$, respectively.
The residues coming from the trivial zeros of $\zeta(s)$,
occur at $s=-2k$ and $s=-2k+1$, for $k\geq 1$, and
are equal to $\frac{X^{-2k+1}}{2k}$ and 
$\frac{X^{-2k+1}}{-2k+1}$ respectively.
Next we sum over $k$ and use the following:
\begin{equation}
  \frac{1}{2}\log\left(1-\frac{1}{X^2}\right)=-\sum_{k\geq 1}\frac{X^{-2k}}{2k}
\end{equation}
\begin{equation}
  \frac{1}{2}\log\left(1+\frac{2}{X-1}\right)=\sum_{k\geq 1}\frac{X^{-2k+1}}{2k-1}
\end{equation}
Putting this all together and dividing by $X$ gives the result.

The proof for $L(s,\chi)$ is similar, with the main difference being that
$L(s,\chi)$ does not have a pole at $s=1$.
The residues arising from non-trivial zeros $\rho$ of $L(s,\chi)$
are exactly as described in above.

When $\chi(-1)=-1$, the trivial zeros of $L(s,\chi)$ occur $s=-1,-3,-5,\dots$.
There is a double-pole at $s=0$ whose residue equals
$$
  -X\frac{L'(0)}{L(0)} + b_\chi + \log X
  \,,
$$
where $b_\chi$ is the constant term in the Laurent series
expansion of
$L'(s-1,\chi)/L(s-1,\chi)$ at $s=0$.
Having already dealt with $s=0$, the remaining
residues arising from the trivial zeros occur at $s=-2k+1$
and $s=-2k$ for $k\geq 1$,
where the residues are equal to
$\frac{X^{-2k+2}}{2k-1}$ and $\frac{X^{-2k}}{-2k}$ respectively.
Summing these residues over $k\geq 1$ gives the result.

When $\chi(-1)=1$, the trivial zeros of $L(s,\chi)$ occur $s=0,-2,-4,-6,\dots$;
the argument is essentially the same, but the details are different.
The main thing to point out is that our definition of $b_\chi$ is slightly different in this case;
here we denote by $b_\chi$ the constant term in the Laurent expansion of
 $L'(s)/L(s)$ at $s=0$.
We include table of the
residues in all three cases, so that the reader can verify the details.
\end{proof}

{\footnotesize
\begin{table}\label{tab:residues}
\begin{tabular}{c|c|c|c|c|c|c|c}
& $0$ & $1$ & $2$ & $\rho$ & $\rho+1$ & $-2k$ & $-2k+1$\\[1ex]
\hline
$\zeta(s)$ & $\displaystyle \frac{\zeta'(-1)}{\zeta(-1)}-X\frac{\zeta'(0)}{\zeta(0)}$ & $X^2$ &
$\displaystyle\frac{X^2}{2}$ & $\displaystyle-\frac{X^{\rho+1}}{\rho}$ & $\displaystyle\frac{X^{\rho+1}}{\rho+1}$ & 
$\displaystyle\frac{X^{-2k+1}}{2k}$ & $\displaystyle\frac{X^{-2k+1}}{-2k+1}$\\
\hline
$\substack{L(s,\chi)\\\chi(-1)=-1}$
& $\displaystyle   -X\frac{L'(0)}{L(0)} + b_\chi + \log X$ & $0$ &
$0$ & $\displaystyle-\frac{X^{\rho+1}}{\rho}$ & $\displaystyle\frac{X^{\rho+1}}{\rho+1}$ & 
$\displaystyle\frac{X^{-2k}}{-2k}$ & $\displaystyle\frac{X^{-2k+2}}{2k-1}$\\
\hline
$\substack{L(s,\chi)\\\chi(-1)=1}$ & $\displaystyle    \frac{L'(-1)}{L(-1)}-X b_\chi-X\log X$ & $0$ &
$0$ & $\displaystyle-\frac{X^{\rho+1}}{\rho}$ & $\displaystyle\frac{X^{\rho+1}}{\rho+1}$ & 
$\displaystyle\frac{X^{-2k+1}}{2k}$ & $\displaystyle\frac{X^{-2k+1}}{-2k+1}$
\end{tabular}
\caption{Residues of $-G(s)X^s/s$}
\end{table}
}

%

Finally, we will require the following lemma to deal with the quantity $b_\chi$
showing up in the explicit formula.

\begin{lemma}
If $\chi(-1)=-1$, then
$$
b_\chi=\frac{L'(2,\chi)}{L(2,\chi)}
+ \sumrhochi\frac{1}{\rho(\rho+1)} -  \sumrhochi\frac{2}{\rho(2-\rho)}
+\log 2 -1
\,.
$$
If $\chi(-1)=1$, then
$$
b_\chi=\frac{L'(2,\chi)}{L(2,\chi)}
-
\sumrhochi\frac{2}{\rho(2-\rho)}
\,.
$$
\end{lemma}

\begin{proof}
Consider Equation 17 on page 83 of~\cite{MR1790423} at $s=-1$ and $s=2$.
Subtracting these and appealing to known properties of the gamma function,
the results follows.
\end{proof}

  \bibliographystyle{plain}
  \bibliography{refs}

\end{document}